\documentclass[reqno,11pt]{amsart}

\usepackage{verbatim} 
\usepackage{xspace} 
\usepackage{amssymb}
\usepackage{amsmath}
\usepackage{amsthm} %
\usepackage[latin1]{inputenc}
\usepackage{graphicx}
\usepackage{amscd,amssymb,euscript,graphics}
\usepackage{young}
\usepackage{youngtab}

\makeindex

\newcommand{\R}{\ensuremath{\mathbb{R}}}
\newcommand{\C}{\ensuremath{\mathbb{C}}}

\newcommand{\E}{\ensuremath{\mathbb{E}}}

\newcommand{\pr}[2]{\ensuremath{\langle {#1},{#2}\rangle}}
\newcommand{\norm}[1]{\ensuremath{\|#1\|}}

\newtheorem{theorem}{Theorem}[section]

\newtheorem{cor}[theorem]{Corollary}

\newtheorem{obs}[theorem]{Remark}

\newcount\theTime
\newcount\theHour
\newcount\theMinute
\newcount\theMinuteTens
\newcount\theScratch
\theTime=\number\time
\theHour=\theTime
\divide\theHour by 60
\theScratch=\theHour
\multiply\theScratch by 60
\theMinute=\theTime
\advance\theMinute by -\theScratch
\theMinuteTens=\theMinute
\divide\theMinuteTens by 10
\theScratch=\theMinuteTens
\multiply\theScratch by 10
\advance\theMinute by -\theScratch

\def\today{{\number\day\space
 \ifcase\month\or
  January\or February\or March\or April\or May\or June\or
  July\or August\or September\or October\or November\or December\fi
 \space\number\year}}

\begin{document}

\title{A Note on Functional Averages over Gaussian Ensembles}
\author{Gabriel H. Tucci and Maria V. Vega}
\address{G. H. Tucci
  is with Bell Laboratories,
  Alcatel--Lucent, 600 Mountain Ave, Murray Hill, NJ 07974.
  E-mail: gabriel.tucci@alcatel-lucent.com}

\email{gabriel.tucci@alcatel-lucent.com}
  
\address{M. V. Vega
	is with William Paterson University,
	Mathematics Department,
	300 Pompton Rd  Wayne, NJ 07470
	E-mail: vegaveglio@wpunj.edu}

\email{vegavegliom@wpunj.edu}

\noindent 
\begin{abstract}
In this work we find a new formula for matrix averages over the Gaussian ensemble. Let ${\bf H}$ be an $n\times n$ Gaussian random matrix with complex, independent, and identically distributed entries of zero mean and unit variance. Given an $n\times n$ positive definite matrix ${\bf A}$, and a continuous function $f:\R^{+}\to\R$ such that $\int_{0}^{\infty}{e^{-\alpha t}|f(t)|^2\,dt}<\infty$ for every $\alpha>0$, we find a new formula for the expectation $\E[\mathrm{Tr}(f({\bf HAH^{*}}))]$. Taking $f(x)=\log(1+x)$ gives another formula for the capacity of the MIMO communication channel, and taking $f(x)=(1+x)^{-1}$ gives the MMSE achieved by a linear receiver.
\end{abstract}

\maketitle



\section{Introduction}%

\noindent Random matrix theory was introduced to the theoretical physics community by Wigner in his work on nuclear physics in the 1950s (\cite{Wigner1, Wigner2}). Since that time, the subject is an important and active research area in mathematics and it finds applications in fields as diverse as the Riemann conjecture, physics, chaotic systems, multivariate statistics, wireless communications, signal processing, compressed sensing and information theory. In the last decades, a considerable amount of work has emerged in the communications and information theory on the fundamental limits of communication channels that makes use of results in random matrix theory \cite{Mehta, Ver, Guionnet}. For this reason, computing averages over certain matrix ensembles becomes extremely important in many situations. To be more specific, consider the well known case of the single user MIMO channel with multiple transmit and receive antennas. Denoting the number of transmitting antennas by $t$ and the number of receiving antennas by $r$, the channel model is 
$$
{\bf y}={\bf Hu}+{\bf n},
$$
where ${\bf u}\in\C^{t}$ is the transmitted vector, ${\bf y}\in\C^{r}$ is the received vector, ${\bf H}$ is a $r\times t$ complex matrix and ${\bf n}$ is the zero mean complex Gaussian vector with independent, equal variance entries. We assume that $\E({\bf nn}^{*})={\bf I}_{r}$, where $(\cdot)^{*}$ denotes the complex conjugate transpose. It is reasonable to put a power constraint 
$$
\E({\bf u^{*}u})=\E(\mathrm{Tr}({\bf uu^{*}}))\leq P,
$$ 
where $P$ is the total transmitted power. The signal to noise ratio, denoted by $\mathrm{snr}$, is defined as the quotient of the signal power and the noise power and in this case is equal to $P/r$.

\vspace{0.1cm}
\noindent Recall that if ${\bf A}$ is an $n\times n$ Hermitian matrix then there exists ${\bf U}$ unitary and ${\bf D}=\mathrm{diag}(d_{1},\ldots,d_{n})$ such that ${\bf A}={\bf UDU^{*}}$. Given a continuous function $f$ we define $f({\bf A})$ as 
$$
f({\bf A})={\bf U}\mathrm{diag}(f(d_1),\ldots,f(d_{n})){\bf U^{*}}.
$$

\par Naturally, the simplest example is the one where ${\bf H}$ has independent and identically distributed (i.i.d.) Gaussian entries, which constitutes the canonical model for the single user narrow band MIMO channel. It is known that the capacity of this channel is achieved when ${\bf u}$ is a complex Gaussian zero mean and covariance $\mathrm{snr}\,{\bf I}_{t}$ vector (see for instance \cite{Emre, Ver}). For the fast fading channel, assuming statistical channel state information at the transmitter, the ergodic capacity is given by 
\begin{equation}
\E\Big[\log\det({\bf I}_{r}+\mathrm{snr}{\bf H}{\bf H}^{*})\Big]=\E\Big[ \mathrm{Tr}\log({\bf I}_{r}+\mathrm{snr}{\bf H}{\bf H}^{*})\Big],
\end{equation}
where in the last equality we use the fact that $\mathrm{Tr}\log(\cdot)=\log\det(\cdot)$. We refer the reader to \cite{Emre} or \cite{Ver} for more details on this.

\vspace{0.1cm}
\noindent  Another important performance measure is the minimum mean square error (MMSE) achieved by a linear receiver, which determines the maximum achievable output signal to interference and noise ratio (SINR). For an input vector ${\bf x}$ with i.i.d. entries of zero mean and unit variance the MSE at the output of the MMSE receiver is given by
\begin{equation}
\min_{{\bf M}\in\C^{t\times r}}\E\Big[\norm{{\bf x-My}}^2\Big] = \E\Big[\mathrm{Tr}\Big({\bf I}_{t}+\mathrm{snr}{\bf H}^{*}{\bf H}\Big)^{-1}\Big],
\end{equation}
where the expectation on the left hand side is over both the vectors ${\bf x}$ and the random matrices ${\bf H}$, while the right hand side is over ${\bf H}$ only. We refer to \cite{Ver} for more details on this.

\vspace{0.1cm}
\noindent There is a big literature and history of work on averages over Gaussian ensembles; see for instance \cite{Emre, Simon, Ver, Alvo, marzetta, Guionnet, Mehta, McKay, Sil1, Debbah, Marchenko, Edelman} and references therein. In \cite{Emre} the capacity of the Gaussian channel was computed as an improper integral. This integral is difficult to compute and asymptotic and simulation results are provided. In \cite{Edelman, Debbah,Alvo, Sil1, Sil2} several asymptotic results for large complex Gaussian random matrices are studied in connection with wireless communication and information theory. In \cite{Alvo} many aspects of correlated Gaussian matrices are addressed, in particular the capacity of Rayleigh channel was computed as the number of antennas increases to infinity. The books \cite{Ver, Mehta, Guionnet} are excellent introductions to random matrix theory and their applications to physics and information theory. In \cite{McKay} the spectral eigenvalue distribution for a random infinite $d$-regular graph was computed.  

\vspace{0.1cm}
\noindent The typical approach in computing averages over random matrices is to consider the asymptotic behavior as the size of the matrix increases to infinity. In this work we contribute to this area by providing a unified framework to express the ergodic mutual information, the MSE at the output of the MMSE decoder and other types of functionals of a single user MIMO channel, when the number of transmitting and receiving antennas are equal and finite. We do not rely on asymptotic results as the number of antennas increases. The results shown in this work are new and novel to the best knowledge of the author and they were not discovered before.

\vspace{0.1cm}
\noindent In Section \ref{sec_prelim}, we present some preliminaries in Schur polynomials that are later used in this work. In Section \ref{sec_main}, we prove the main result of the paper, Theorem \ref{main}. This Theorem provides a new formula for the expectation 
\begin{equation}
\E\Big[\mathrm{Tr}\big(f({\bf HAH}^{*})\big)\Big],
\end{equation}
where ${\bf A}$ is a positive definite matrix and $f$ a continuous function such that 
$$
\int_{0}^{\infty}{e^{-\alpha t}|f(t)|^2\,dt}<\infty
$$ 
for every $\alpha>0$. Notice that, as previously stated, taking $f(x)=\log(1+x)$ gives another formula for the capacity of the MIMO communication channel, and taking $f(x)=(1+x)^{-1}$ gives the MMSE achieved by a linear receiver. We also discuss some applications and present some examples. 

\section{Schur Polynomials Preliminaries}\label{sec_prelim}

\noindent A symmetric polynomial is a polynomial $P(x_1,x_2,\ldots, x_n)$ in $n$ variables such that if any of the variables are interchanged one obtains the same polynomial. Formally, $P$ is a symmetric polynomial if for any permutation $\sigma$ of the set $\{1,2,\ldots, n\}$ one has 
$$
P(x_{\sigma(1)},x_{\sigma(2)},\ldots, x_{\sigma(n)})=P(x_1,x_2,\ldots, x_{n}).
$$
Symmetric polynomials arise naturally in the study of the relation between the roots of a polynomial in one variable and its coefficients, since the coefficients can be given by a symmetric polynomial expressions in the roots. Symmetric polynomials also form an interesting object by themselves. The resulting structures, and in particular the ring of symmetric functions, are of great importance in combinatorics and in representation theory (see for instance \cite{Fulton, Muir, Mac, Sagan} for more on details on this topic).

\vspace{0.1cm}
\noindent The Schur polynomials are certain symmetric polynomials in $n$ variables. This class of polynomials is very important in representation theory since they are the characters of irreducible representations of the general linear groups. The Schur polynomials are indexed by partitions. A partition of a positive integer $n$, also called an integer partition, is a way of writing $n$ as a sum of positive integers. Two partitions that differ only by the order of their summands are considered to be equal. Therefore, we can always represent a partition $\lambda$ of a positive integer $n$ as a non-increasing sequence of $n$ non-negative integers $d_i$ such that 
$$
\sum_{i=1}^{n}{d_i}=n \hspace{0.5cm} \text{with} \hspace{0.5cm} d_{1}\geq d_2\geq d_{3}\geq\ldots\geq d_{n}\geq 0.
$$ 
Notice that some of the $d_i$ could be zero. Integer partitions are usually represented by the so called Young's diagrams (also known as Ferrers' diagrams). A Young diagram is a finite collection of boxes, or cells, arranged in left-justified rows, with the row lengths weakly decreasing (each row has the same or shorter length than its predecessor). Listing the number of boxes on each row gives a partition $\lambda$ of a non-negative integer $n$, the total number of boxes of the diagram. The Young diagram is said to be of shape $\lambda$, and it carries the same information as that partition. For instance, below we can see the Young diagram corresponding to the partition $(5,4,1)$ of the number 10.
$$
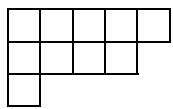
\yng(5,4,1)
$$
Given a partition $\lambda$ of $n$
$$   
n = d_1 + d_2 + \cdots + d_n \,\,\, : \,\,\, d_1 \geq d_2 \geq \cdots \ge d_n\geq 0
$$
the following functions are alternating polynomials (in other words they change sign under any transposition of the variables):
\begin{eqnarray*}
a_{(d_1,\ldots,d_n)}(x_1, \ldots , x_n) &=& \det \left[ \begin{matrix} x_1^{d_1} & x_2^{d_1} & \dots & x_n^{d_1} \\ x_1^{d_2} & x_2^{d_2} & \dots & x_n^{d_2} \\ \vdots & \vdots & \ddots & \vdots \\ x_1^{d_n} & x_2^{d_n} & \dots & x_n^{d_n} \end{matrix} \right] \\
& = &\sum_{\sigma\in S_n}\epsilon(\sigma)x_{\sigma(1)}^{d_1}\cdots x_{\sigma(n)}^{d_n}
\end{eqnarray*}
where $S_{n}$ is the permutation group of the set $\{1,2,\ldots,n\}$. Since they are alternating, they are all divisible by the Vandermonde determinant 
$$
\Delta(x_1,\ldots,x_n)=\prod_{1 \leq j < k \leq n} (x_j-x_k).
$$ 
The Schur polynomial associated to $\lambda$ is defined as the ratio:
$$
s_{\lambda} (x_1, x_2, \dots , x_n) = \frac{ a_{(d_1+n-1, d_2+n-2, \dots , d_n+0)} (x_1, \dots , x_n)} {\Delta(x_1,\ldots,x_n)}. 
$$
This is a symmetric function because the numerator and denominator are both alternating, and a polynomial since all alternating polynomials are divisible by the Vandermonde determinant (see \cite{Fulton,Mac, Sagan} for more details here). For instance,
$$
s_{(2,1,1)} (x_1, x_2, x_3) = x_1 \, x_2 \, x_3 \, (x_1 + x_2 + x_3) 
$$
and
$$
s_{(2,2,0)} (x_1, x_2, x_3) =  x_1^2 \, x_2^2 + x_1^2 \, x_3^2 + x_2^2 \, x_3^2 + x_1^2 \, x_2 \, x_3 +  x_1 \, x_2^2 \, x_3 + x_1 \, x_2 \, x_3^2.  
$$

\noindent Another definition we need for the next Section is the so called hook length, $\mathrm{hook}(x)$, of a box $x$ in Young diagram of shape $\lambda$. This is defined as the number of boxes that are in the same row to the right of it plus those boxes in the same column below it, plus one (for the box itself). As an example, below we show the hook lengths of the partition $(5,4,1)$. The product of the hook's length of a partition is the product of the hook lengths of all the boxes in the partition. 
$$
\begin{Young}
7 & 5 & 4 & 3 & 1 \cr
5 & 3 & 2 & 1 \cr
1 \cr
\end{Young}
$$
\noindent We recommend the interested reader to consult \cite{Fulton,Mac,Sagan} for more details and examples on this topic. 

\section{Averages over Gaussian Ensembles}\label{sec_main}

\noindent Let $M_{n}$ be the set of all $n\times n$ complex matrices and ${\bf U}_{n}$ the set of $n\times n$ unitary complex matrices. Let $d{\bf H}$ be the Lebesgue measure on $M_{n}$ and let 
$$
d\nu({\bf H})=\pi^{-n^2}\,\mathrm{exp}\Big(-\mathrm{trace}({\bf H}^{*}{\bf H})\Big)\,d{\bf H}
$$ 
be the Gaussian measure on $M_{n}$. This is the induced measure by the Gaussian random matrix with complex independent and identically distributed entries with zero mean and unit variance in the set of matrices,  when this is represented as an Euclidean space of dimension $2n^2$. Note that this probability measure is left and right invariant under unitary multiplication (i.e., $d\nu({\bf HU})=d\nu({\bf UH})=d\nu({\bf H})$ for every unitary ${\bf U}$). The following Theorem can be found on page 447 of \cite{Mac}.

\begin{theorem}\cite{Mac}\label{teo_mac}
For all Hermitian $n\times n$ matrices ${\bf A,B}$ and every partition $\lambda$
\begin{equation}
\int_{M_{n}}{s_{\lambda}({\bf AH^{*}BH})\,d\nu({\bf H})}=h(\lambda)s_{\lambda}({\bf A})s_{\lambda}({\bf B}),
\end{equation}
where $h(\lambda)$ is the product of the hook--lengths of $\lambda$.
\end{theorem} 

\noindent Denote by $(m-k,1^{k})$ the partition $(m-k,1,1,\ldots,1)$ with $k$ ones. It is a well known fact in matrix theory (see \cite{Fulton} or \cite{Mac}) that for every Hermitian $n\times n$ matrix ${\bf A}$ and for every integer $m$ 
\begin{equation}\label{traceschur}
\mathrm{Tr}({\bf A}^{m})=\sum_{k=0}^{n-1}{(-1)^{k}s_{(m-k,1^{k})}({\bf A})}. 
\end{equation}
Note that for the case $1\leq m<n$, even though the sum is up to the $n-1$ term, all the terms between $\min\{n,m\}$ and $n-1$ are zero. In particular, 

\begin{itemize}
\item $\mathrm{Tr}({\bf A})=s_{(1)}({\bf A})$,
\item $\mathrm{Tr}({\bf A}^2)=s_{(2,0)}({\bf A})-s_{(1,1)}({\bf A})$,
\item $\mathrm{Tr}({\bf A}^{3})=s_{(3,0)}({\bf A})-s_{(2,1)}({\bf A})+s_{(1,1,1)}({\bf A})$,
\item $\mathrm{Tr}({\bf A}^{4})=s_{(4,0)}({\bf A})-s_{(3,1)}({\bf A})+s_{(2,1,1)}({\bf A})-s_{(1,1,1,1)}({\bf A})$.
\end{itemize}

\noindent The constant $s_{(m-k,1^{k})}({\bf I}_{p})$ is equal to
$$
s_{(m-k,1^{k})}({\bf I}_{p})=\frac{(m+p-(k+1))!}{k!(p-(k+1))!(m-(k+1)!m}
$$ 
(see \cite{Mac} for a proof of this formula). Therefore,
\begin{equation}\label{const}
\frac{s_{(m-k,1^{k})}({\bf I}_{p})}{s_{(m-k,1^{k})}({\bf I}_{n})}=\frac{(m+p-(k+1))!}{(m+n-(k+1))!}\cdot\frac{(n-(k+1))!}{(p-(k+1))!}.
\end{equation}

\noindent For every $\alpha>0$ let us define the following class of functions 
\begin{equation}
L^{2}_{\alpha}: = \Big\{f:\R^{+}\to\R\,:\text{measurable such that}\,\,\, \int_{0}^{\infty}{e^{-\alpha t}\,|f(t)|^2\,dt<\infty}\Big\}.
\end{equation}
This is a Hilbert space with respect to the inner product $\pr{f}{g}_{\alpha}=\int_{0}^{\infty}{e^{-\alpha t}\,f(t)g(t)\,dt}$. Moreover, polynomials are dense with respect to this norm (see Chapter 10 in \cite{Laguerre}). Let $\mathcal{A}_{\alpha}$ be the set of continuous functions in $L^{2}_{\alpha}$ and let $\mathcal{A}$ be the intersection of all the $\mathcal{A}_{\alpha}$, 
$$
\mathcal{A}=\bigcap_{\alpha>0}{\mathcal{A}_{\alpha}}.
$$
Note that the family $\mathcal{A}$ is a very rich family of functions. For instance, all functions that do not grow faster than polynomials belong to this family. In particular, $f(t)=\log(1+t)\in\mathcal{A}$.

\begin{theorem}\label{main}
Let ${\bf A}$ be an $n\times n$ positive definite matrix and let $\{d_{1},\ldots,d_{n}\}$ be the set of eigenvalues of ${\bf A}$. Assume that all the eigenvalues are different. Then for every $f\in\mathcal{A}$ 
we have that 
\begin{equation}\label{trf}
\int_{M_{n}}{\mathrm{Tr}\Big(f({\bf H^{*}AH}))\Big)\,d\nu({\bf H})}=\frac{1}{\det(\Delta({\bf D}))}\sum_{k=0}^{n-1}{\det({\bf T}_{k})},
\end{equation}
where $\Delta({\bf D})$ is the Vandermonde matrix associated with the matrix ${\bf D}=\mathrm{diag}(d_1,\ldots,d_n)$ and ${\bf T}_{k}$ is the matrix constructed by replacing the $(k+1)$ row of $\Delta({\bf D})$ ($\{d_{i}^{n-(k+1)}\}_{i=1}^{n}$) by 
$$
\frac{1}{(n-(k+1))!}\{f_{k}(d_{i})\}_{i=1}^{n}
$$ 
where
$$
f_{k}(x):=\int_{0}^{\infty}{e^{-t}(tx)^{n-(k+1)}f(tx)\,dt}.
$$
\end{theorem}

\begin{proof}
\noindent First, we will prove the Theorem for polynomials. Let $p$ and $q$ be two polynomials. It is clear that 
$$
\mathrm{Tr}\big((p+q)({\bf H^{*}AH})\big)=\mathrm{Tr}\big(p({\bf H^{*}AH})\big)+\mathrm{Tr}\big(q({\bf H^{*}AH})\big)
$$ 
and $(p+q)_{k}=p_{k}+q_{k}$ for every $k=0,\ldots,n-1$. Therefore, both sides of the Equation (\ref{trf}) are linear and it is enough to prove the Theorem for the case $p(x)=x^{m}$ with $m\geq 0$. Using Theorem \ref{teo_mac} and Equation (\ref{traceschur}) we see that for every positive definite $n\times n$ matrix ${\bf A}$, the average
$$
\int_{M_{n}}{\mathrm{Tr}\Big(({\bf H^{*}AH})^m\Big)\,d\nu({\bf H})} = \sum_{k=0}^{n-1}{(-1)^k h(\lambda_k)s_{\lambda_k}({\bf A})s_{\lambda_{k}}({\bf I}_n)}, 
$$
where $\lambda_{k}$ is the partition $(m-k,1^{k})$. It is well known (see \cite{Fulton}) that for every partition $\lambda=(\lambda_{1},\ldots,\lambda_{n})$ 
\begin{equation}
s_{\lambda}({\bf I}_n)=\prod_{1\leq i\leq j\leq n}{\frac{\lambda_i-\lambda_j+j-i}{j-i}}.
\end{equation}
Therefore, we can deduce that
\begin{equation}\label{cons}
s_{\lambda_k}({\bf I}_n)=\frac{1}{m}\cdot\frac{(m+n-(k+1))!}{k!\,(n-(k+1))!\,(m-(k+1))!}.
\end{equation}
\noindent We can see by direct examination that the hook--length of the partition $\lambda_{k}$ is equal to 
$$
h(\lambda_{k})=k!\,(m-(k+1))!\,m.
$$
Hence, 
$$
s_{\lambda_k}({\bf I}_n)h(\lambda_k)=\frac{(m+n-(k+1))!}{(n-(k+1))!}.
$$
Since ${\bf A}$ is a positive definite matrix, by the spectral Theorem there exists ${\bf U}$ unitary and ${\bf D}=\mathrm{diag}(d_1,\ldots,d_{n})$ diagonal such that ${\bf A}={\bf UDU^{*}}$. Note that the $d_{i}$ are the eigenvalues of ${\bf A}$. By definition of the Schur polynomials 
$$
s_{\lambda_{k}}({\bf A})=s_{\lambda_{k}}({\bf D})=\frac{\det({\bf S}_k)}{\det(\Delta({\bf D}))},
$$ 
where $\Delta({\bf D})$ is the Vandermonde matrix associated with the sequence $\{d_{i}\}_{i=1}^{n}$ and ${\bf S}_{k}$ is a matrix whose $i$--th column is equal to 
$$
\begin{pmatrix}
d_{i}^{n-1+m-k} \\ 
d_{i}^{n-2+1} \\
d_{i}^{n-3+1} \\
\vdots \\ 
d_{i}^{n-(k+1)+1} \\ 
d_{i}^{n-(k+2)}\\
\vdots \\
d_{i}^{n-(n-1)}\\
1\\
\end{pmatrix}.
$$
It is easy to see that after $k$ transpositions of the rows of the matrix ${\bf S}_{k}$ we obtain a new matrix ${\bf H}_{k}$ whose $i$--th column is equal to 
$$
\begin{pmatrix}
d_{i}^{n-1} \\ 
d_{i}^{n-2} \\
d_{i}^{n-3} \\
\vdots \\ 
d_{i}^{n-k} \\ 
d_{i}^{n+m-(k+1)}\\
d_{i}^{n-(k+2)}\\
\vdots \\
d_{i}^{n-(n-1)}\\
1\\
\end{pmatrix}.
$$
This matrix is equal to the matrix $\Delta({\bf D})$ except for the $(k+1)$ row, $\{d_{i}^{n-(k+1)}\}_{i=1}^{n}$, which is substituted by   
the row $\{d_{i}^{n+m-(k+1)}\}_{i=1}^{n}$. Note also that 
$$
\det({\bf S}_{k})=(-1)^{k}\det({\bf H}_k).
$$
Therefore,
$$
\int_{M_{n}}{\mathrm{Tr}\Big(({\bf H^{*}AH})^m\Big)\,d\nu({\bf H})} = \frac{1}{\det(\Delta({\bf D}))}\sum_{k=0}^{n-1}{\frac{(m+n-(k+1))!}{(n-(k+1))!}\cdot\det({\bf H}_k)}.
$$
Using the fact that $\int_{0}^{\infty}{e^{-t}\,t^p\,dt}=p!$ and the definition of $p_{k}(x)$ for the case $p(x)=x^{m}$ we see that
\begin{equation}
p_{k}(x):= \int_{0}^{\infty}{e^{-t}(tx)^{n+m-(k+1)}\,dt} = (m+n-(k+1))\,!\,\, x^{m+n-(k+1)}.
\end{equation}
Therefore, our claim holds and we have proven the result for all polynomials. Now consider $f\in\mathcal{A}$ and let $\beta$ be the maximum eigenvalue, i.e., $\beta=\max\{d_1,\ldots,d_{n}\}$. Define $\alpha=1/\beta$. Since $f\in\mathcal{A}$, then $f\in\mathcal{A}_\alpha$ and let $\{p^{(r)}\}_{r\geq 1}$ be a sequence of polynomials such that $\norm{f-p^{(r)}}_{\alpha}\to 0$. Let ${\bf T}_{k}^{(n)}$ be the matrix constructed by replacing the $(k+1)$ row of $\Delta({\bf D})$ ($\{d_{i}^{n-(k+1)}\}_{i=1}^{n}$) by 
$$
\frac{1}{(n-(k+1))!}\{p_{k}^{(r)}(d_{i})\}_{i=1}^{n},
$$ 
where
$$
p_{k}^{(r)}(x):=\int_{0}^{\infty}{e^{-t}(tx)^{n-(k+1)}p^{(r)}(tx)\,dt}.
$$
Let ${\bf T}_{k}$ be the matrix constructed by replacing the $(k+1)$ row of $\Delta({\bf D})$ by 
$$
\frac{1}{(n-(k+1))!}\{f_{k}(d_{i})\}_{i=1}^{n},
$$ 
where
$$
f_{k}(x):=\int_{0}^{\infty}{e^{-t}(tx)^{n-(k+1)}f(tx)\,dt}.
$$
To prove that Equation (\ref{trf}) holds it is enough to prove that 
$$
\det({\bf T}_{k}^{(n)})\to \det({\bf T}_{k})
$$ 
as $n\to\infty$ for every $k=0,1,\ldots,n-1$. For this, it is enough to prove that $p_{k}^{(r)}(d_i)\to f_{k}(d_i)$ for every $k$ and every $i=1,2,\ldots,n$. Note that
\begin{eqnarray*}
\lefteqn{|f_{k}(d_i)-p_{k}^{(r)}(d_i)|} \\ 
& = & \int_{0}^{\infty}{e^{-t}(td_i)^{n-(k+1)}|f(td_i)-p^{(r)}(td_i)|\,dt}\\
& \leq & d_{i}^{n-(k+1)}\sqrt{(2(n-(k+1)))!} \cdot  \Big( \int_{0}^{\infty}{e^{-t}|f(td_i)-p^{(r)}(tdi)|^2\,dt}\Big)^{\frac{1}{2}}\\
& = & d_{i}^{n-(k+\frac{3}{2})}\sqrt{(2(n-(k+1)))!}\cdot \Big(\int_{0}^{\infty}{e^{-\frac{t}{d_i}}|f(t)-p^{(r)}(t)|^2\,dt}\Big)^{\frac{1}{2}},
\end{eqnarray*}
where we use Cauchy-Schwartz for the second inequality and change of variable for the last one. Now, by construction the sequence $\{p^{(r)}\}$ satisfies
$$
\lim_{n\to\infty}{\norm{f-p^{(r)}}}_{\alpha}^{2}=\lim_{n\to\infty}{\int_{0}^{\infty}{e^{-\alpha t}|f(t)-p^{(r)}(t)|^2\,dt}}=0
$$
and $\alpha\leq d_{i}^{-1}$. Hence, we see that 
$$
\lim_{n\to\infty}{|f_{k}(d_i)-p_{k}^{(r)}(d_i)|}=0
$$ 
finishing the proof.
\end{proof}

\begin{obs}
We would like to observe that the case when not all the eigenvalues are different can be treated as above by perturbing of the original eigenvalues and applying a subsequent limit. We present an instance of this situation in Corollary \ref{ex}.
\end{obs}

\noindent As a consequence we have a new formula for the capacity of the MIMO communication channel and for the MMSE described in the introduction.
\vspace{0.3cm}
\begin{cor}
Let ${\bf A}$ be as in Theorem \ref{main}. Then
\begin{equation}
\int_{M_{n}}{\mathrm{Tr}\Big(\log({\bf I}_{n}+{\bf H^{*}AH})\Big)\,d\nu({\bf H})} = \frac{1}{\det(\Delta({\bf D}))}\sum_{k=0}^{n-1}{\det({\bf T}_{k})},
\end{equation}
where ${\bf T}_{k}$ is the matrix constructed by replacing the $(k+1)$ row of $\Delta({\bf D})$ ($\{d_{i}^{n-(k+1)}\}_{i=1}^{n}$) by 
\begin{equation*}
\Bigg\{ \frac{1}{(n-(k+1))!}\int_{0}^{\infty}{e^{-t}(td_{i})^{n-(k+1)}\log(1+td_i)\,dt}\Bigg\}_{i=1}^{n}.
\end{equation*}
\end{cor}

\vspace{0.3cm}
\begin{cor}
Let ${\bf A}$ be as in theorem \ref{main}. Then
\begin{equation}
\int_{M_{n}}{\mathrm{Tr}\Big(({\bf I}_{n}+{\bf H^{*}AH})^{-1}\Big)\,d\nu({\bf H})} = \frac{1}{\det(\Delta({\bf D}))}\sum_{k=0}^{n-1}{\det({\bf T}_{k})},
\end{equation}
where ${\bf T}_{k}$ is the matrix constructed by replacing the $(k+1)$ row of $\Delta({\bf D})$ ($\{d_{i}^{n-(k+1)}\}_{i=1}^{n}$) by 
\begin{equation*}
\Bigg\{ \frac{1}{(n-(k+1))!}\int_{0}^{\infty}{e^{-t}(td_{i})^{n-(k+1)}(1+td_i)^{-1}\,dt}\Bigg\}_{i=1}^{n}.
\end{equation*}
\end{cor}

\noindent As an application let us compute explicitly the two dimensional case for the capacity.

\begin{cor}\label{ex}
Let ${\bf A}$ be an Hermitian $2\times 2$ matrix with eigenvalues $d_{1}$ and $d_{2}$. If $d_1\neq d_2$ then 
$$
\int_{M_{2}}{\mathrm{Tr}\Big(\log({\bf I}_{2}+{\bf H^{*}AH})\Big)\,d\nu({\bf H})} = \frac{f_{0}(d_1)-f_{0}(d_2)+d_{1}f_{1}(d_2)-d_{2}f_{1}(d_{1})}{d_{1}-d_{2}},
$$
where $f_{0}(d_{i})=\int_{0}^{\infty}{e^{-t}td_{i}\log(1+td_{i})\,dt}$ and $f_{1}(d_{i})=\int_{0}^{\infty}{e^{-t}\log(1+td_{i})\,dt}$. If $d_1=d_2=d$ then 
$$
\int_{M_{2}}{\mathrm{Tr}\Big(\log({\bf I}_{2}+d\cdot{\bf H^{*}H})\Big)\,d\nu({\bf H})} =  \int_{0}^{\infty}{e^{-t}\Big[(1+t)\log(1+td)+\frac{td(t-1)}{1+td}\Big]\,dt}.
$$
\end{cor}

\vspace{0.3cm}
\begin{proof}
The case $d_{1}\neq d_{2}$ is a direct application of theorem \ref{main} for $n=2$ and $f(x)=\log(1+x)$. For the case $d_{1}=d_{2}=d$ then both the top and the bottom vanish and we have to take the limit of $d_{1}=d+\epsilon$ and $d_{2}=d$ as $\epsilon\to 0$. More precisely,
$$
\lim_{\epsilon\to 0}{\frac{f_{0}(d+\epsilon)-f_{0}(d)}{\epsilon}}=\int_{0}^{\infty}{e^{-t}\Big[t\log(1+td)+\frac{t^2d}{1+td}\Big]dt}
$$
and 
$$
\lim_{\epsilon\to 0}{\frac{(d+\epsilon)f_{1}(d)-df_{1}(d+\epsilon)}{\epsilon}} = \int_{0}^{\infty}{e^{-t}\Big[(1+d)\log(1+td)-\frac{td}{1+td}\Big]\,dt}.
$$
Putting all the pieces together we finish the proof.
\end{proof}

\noindent Analogously, we can compute explicitly the moments for the two dimensional case.

\begin{theorem}
Let ${\bf A}$ be an Hermitian $2\times 2$ matrix with eigenvalues $d_{1}$ and $d_{2}$ and let $m\geq 1$. If $d_1\neq d_2$ then 
$$
\int_{M_{2}}{\mathrm{Tr}\Big(({\bf H^{*}AH})^m\Big)\,d\nu({\bf H})} = m!\Bigg( (m+1)\frac{d_{1}^{m+1}-d_{2}^{m+1}}{d_1-d_2}+\frac{d_{1}d_{2}^{m}-d_{2}d_{1}^{m}}{d_{1}-d_{2}}\Bigg).
$$
If $d_1=d_2=d$ then 
$$
\int_{M_{2}}{\mathrm{Tr}\Big(({\bf H^{*}AH})^m\Big)\,d\nu({\bf H})}=m!\,(m^2+m+2)d^{m}.
$$
\end{theorem}

\section{Conclusion}

\noindent Using results on random matrix theory and representation theory, in particular Schur polynomials, we prove a new formula for the average of functionals over the Gaussian ensemble. In particular, this gives another formula for the capacity of the MIMO Gaussian channel and the MMSE achieved by a linear receiver.

\end{document}